\documentclass[a4paper]{article}
\usepackage[utf8]{inputenc}
\usepackage{lmodern}
\usepackage[T1]{fontenc}

\usepackage{mathtools, amssymb, amsthm}
\usepackage{enumitem}
\usepackage{bbm,dsfont}
\usepackage{color}
\numberwithin{equation}{section}
\allowdisplaybreaks

\newcommand{\MR}{\textit{MR}}

\newcommand{\1}{\mathds{1}}
\newcommand{\K}{\mathds{K}}
\newcommand{\R}{\mathds{R}}
\newcommand{\C}{\mathds{C}}

\newcommand{\A}{\mathcal{A}}
\newcommand{\B}{\mathcal{B}}

\newcommand{\D}{\mathcal{D}}
\newcommand{\V}{\mathcal{V}}

\renewcommand{\L}{\mathcal{L}}
\renewcommand{\d}{\mathrm{d}}
\renewcommand{\Re}{\operatorname{Re}}
\renewcommand{\Im}{\operatorname{Im}}
\newcommand{\fra}{\mathfrak{a}}
\newcommand{\frb}{\mathfrak{b}}
\newcommand{\frc}{\mathfrak{c}}
\renewcommand{\mid}{\, \vert \,}

\DeclarePairedDelimiter\abs{\lvert}{\rvert}
 \DeclarePairedDelimiter\norm{\lVert}{\rVert} 

\theoremstyle{plain}
\newtheorem{theorem}{Theorem}[section]
\newtheorem{proposition}[theorem]{Proposition}
\newtheorem{lemma}[theorem]{Lemma}

\theoremstyle{definition}

\theoremstyle{remark}
\newtheorem{remark}[theorem]{Remark}

\begin{document}
\title{Maximal Regularity for Non-Autonomous  Second Order Cauchy Problems}
\author{
    Dominik Dier,
    El Maati Ouhabaz\footnote{Corresponding author.}
}



\maketitle

\begin{abstract}\label{abstract}
We consider  non-autonomous wave equations 
\[		\left\{
		\begin{aligned}
			&\ddot u(t) + \B(t)\dot u(t) + \A(t)u(t) = f(t) \quad t\text{-a.e.}\\
			&u(0)=u_0,\,  \dot u(0) = u_1.
		\end{aligned}
		\right.
	\]
where the operators  $\A(t)$ and $\B(t)$ are  associated with time-dependent sesquilinear forms  $\fra(t,.,.)$  and $\frb$ defined on a Hilbert space $H$ with the same domain $V$. The initial values  satisfy $ u_0 \in V$ and $u_1 \in H$. 
We  prove well-posedness and  maximal regularity for the solution both in the spaces $V'$ and $H$.  
We apply the results to non-autonomous Robin-boundary conditions and also use maximal regularity to solve a quasilinear problem.
\end{abstract}

\medskip


\section{Introduction}\label{section:introduction}

The present paper is a continuation of \cite{ADLO} which is devoted to  maximal regularity 
for first order  non-autonomous evolution equations governed by forms.  Here we address the problem of maximal regularity
for non-autonomous  second order problems. \\
We consider Hilbert spaces $H$ and $V$ such that $V$ is continuously embedded into $H$ and two families of sesquilinear  forms
\[
	\fra\colon [0,T]\times V\times V \to \C, \quad \frb\colon [0,T]\times V\times V \to \C
\]
such that 
$\fra(.,u,v)\colon [0,T]\to\C$,  $\frb(.,u,v)\colon [0,T]\to\C$ are measurable for all $u,v \in V$, 
\[
	\abs{ \fra(t,u,v) } \le M \norm{u}_V \norm{v}_V \quad (t\in[0,T]), 
\] 
 and  
\[
	\Re \fra (t,u,u) + w \norm{u}_H^2  \ge \alpha \|u\|^2_V \quad (u\in V, t\in [0,T]) 
\] 
where $M \ge 0$, $w \in \R$, and $\alpha > 0$ are constants.  We assume also that $\frb$ satisfies the same properties. 
For fixed $t\in [0,T]$, we denote by  $\A(t), \B(t) \in \L(V,V^\prime)$ the operators associated with the forms
$\fra(t, .,.)$ and $\frb(t,.,.)$, respectively. Given a function $f$ defined on $ [0,T]$ with values either in $H$ or in $V'$ and consider 
the second order evolution equation
\begin{equation}\label{prob1}
		\left\{
		\begin{aligned}
			&\ddot u(t) + \B(t)\dot u(t) + \A(t)u(t) = f(t) \quad t\text{-a.e.}\\
			&u(0)=u_0,\,  \dot u(0) = u_1.
		\end{aligned}
		\right.
	\end{equation}
with initial values $u_0 \in V$ and $u_1 \in H$.  This is a damped non-autonomous wave equation.  The equation without the factor $\dot u$, i.e., 
\begin{equation}\label{prob2}
		\left\{
		\begin{aligned}
			&\ddot u(t)  + \A(t)u(t) = f(t) \quad t\text{-a.e.}\\
			&u(0)=u_0,\,  \dot u(0) = u_1.
		\end{aligned}
		\right.
	\end{equation}
is a non-autonomous wave equation. 

Our aim is to prove well-posedness and maximal regularity for \eqref{prob1} and \eqref{prob2}. We shall prove three main results. The first one concerns maximal regularity in $V'$ for the damped wave equation \eqref{prob1}. We prove that for $u_0 \in V, u_1 \in H$ and $f \in L^2(0,T, V')$ there exists a unique solution 
$u \in H^1(0,T, V) \cap H^2(0,T, V')$. This result  was first proved by Lions \cite[p.\ 151]{Lio61}  by assuming  regularity of  $t \mapsto \fra(t,u,v)$ and $t \mapsto \frb(t,u,v)$ for every fixed $u, v \in V$.  This regularity assumption was removed  in Dautray-Lions \cite[p.\ 667]{DL88}, but taking $f \in L^2(0, T, H)$ and considering mainly symmetric forms. The general case was  given recently by Batty, Chill and Srivastava  \cite{BCS} by  reducing the problem to a first order non-autonomous equation. The result in  \cite{BCS} is stated  in the case $u_0 = u_1 = 0$, only. 
Our proof is different from \cite{BCS} and is inspired by that of   
 Lions \cite{Lio61}. Next we consider maximal regularity  in $H$. This is more delicate and needs extra properties on the forms $\fra$ and $\frb$. We prove that if the forms are symmetric and $t \mapsto \fra(t,u,v)$ and $t \mapsto \frb(t,u,v)$ are piecewise Lipschitz on $[0,T]$ then for $u_0 \in V$, $u_1 \in H$ and $f \in L^2(0,T,H)$ there exists a unique solution $u \in H^1(0,T,V) \cap H^2(0,T,H)$ to  the equation \eqref{prob1}. We also allow some non-symmetric perturbations
 of $\fra$ and $\frb$. The  third result  (Theorem \ref{thm:MR_in_V'}) concerns the wave equation \eqref{prob2}.  We prove that if $\fra$ is symmetric 
and $t \mapsto \fra(t,u,v)$ is Lipschitz on $[0,T]$, 
then for every $u_0 \in V$, $u_1 \in H$ and $f \in L^2(0,T,H)$ there exists a unique solution $u \in L^2(0,T,V) \cap H^1(0,T,H) \cap H^2(0,T,V')$ to the equation \eqref{prob2}. This result is not new and was already proved by Lions \cite[p.\ 150]{Lio61} for the  case $u_0 = 0$ and later  in 
\cite[p.\ 666]{DL88}  for $u_0 \in V$ and $u_1 \in H$.    Theorem \ref{thm:MR_in_V'}  is stated in order to have a complete  picture of maximal regularity for wave equations with or  without damping.  The proof in  \cite{DL88} uses a Galerkin method and sectorial approximation.  The proofs of  the three main theorems   use a representation result  of Lions (see  Theorem \ref{thlions} below) for a given   sesquilinear form $E$  acting on a product of a Hilbert  and pre-Hilbert  spaces $\mathcal H \times \V$.   In each case we have to define the appropriate  spaces $\mathcal H$, $\V$ and the form $E$ to which we apply Theorem 
\ref{thlions}. This idea was already used in \cite{Lio61}. Our choice of the spaces $\mathcal H$, $\V$ and  the form $E$ allow us to sharpen and extend some results from \cite{Lio61} and assume less regularity on $t \mapsto \fra(t,u,v)$ and $t \mapsto \frb(t,u,v)$.

We illustrate our abstract results by two examples. The first one  is a linear damped wave  equation with time dependent Robin boundary conditions. 
The second  is a quasi-linear second order non-autonomous problem.  
The latter is treated by a fixed point argument  but the  implementation of this classical idea 
uses heavily a priori estimates that follow from our  maximal regularity results for linear equations.

\subsection*{Acknowledgment}
Some ideas in this work germinated during a visit  of the second named author at  the University of Ulm in the framework of the Graduate School: 
Mathematical Analysis of Evolution, Information and Complexity financed by the Land Baden-Württemberg 
and during the visit of the first named author at the University of Bordeaux. Both authors thank Wolfgang Arendt for fruitful discussions on the non-autonomous maximal regularity. \\
D.\ Dier is a member of the DFG Graduate School 1100: Modeling,  Analysis and Simulation in Economics.\\
The research of E.\ M.\ Ouhabaz  is partly supported by the ANR project ``Harmonic Analysis at its Boundaries'',  ANR-12-BS01-0013-02.

\section{Preliminaries}
Throughout  this paper, $V$ and $H$ are separable Hilbert spaces over the field  $\K = \C$ or $\R$.
 The   scalar  products of $H$ and $V$  and the corresponding norms will be denoted by  $(. \mid .)_H$, $(. \mid .)_V$, $\norm{.}_H$ and $\norm{.}_V$, respectively. We denote by  $V'$   the antidual of $V$ when  $\K =\C$ and the dual when $\K =\R$. 
 The duality between $V'$ and $V$ is denoted by $\langle ., . \rangle$. Then $\langle u, v \rangle = (u \mid v)_H$ for $u \in H$ and $v \in V$.\\
  We assume that
\[
    V \underset d \hookrightarrow H;
\]
i.e., $V$ is a dense subspace of $H$ such that for some constant
$c_H >0$,
\begin{equation}\label{eq:V_dense_in_H}
    \norm{u}_H \le c_H \norm u _V \quad (u \in V).
\end{equation}
By duality and density of $V$ in $H$ one has
\[
    H \underset d \hookrightarrow V'.
\]
The space $H$ is then  identified with a dense subspace of $V'$ (associating to $u \in H$ the antilinear map  $v \mapsto (u \mid v)_H = \langle u, v \rangle$ 
for $v \in V$).

 Let
\[
    \fra \colon [0, T] \times V \times V \to \K
\]
be a family of sesquilinear and $V$-bounded forms; i.e.\
\begin{equation}\label{eq:a_continuous}
    \abs{\fra(t, u,v)} \le M \norm u _V \norm v _V \quad (u,v \in V, t \in [0,T])
\end{equation}
for some constant $M$, such that $\fra(.,u,v)$ is measurable for all $u,v \in V$.  
We shall call $\fra$ satisfying the above properties 
a  \emph{$V$-bounded non-autonomous sesquilinear form}.
Moreover we say that  $\fra$ is \emph{quasi-coercive} if 
there exist constants $\alpha >0$, $\omega \in \R$ such that
\begin{equation}\label{eq:H-elliptic}
    \Re \fra(t, u,u) + \omega \norm u_H^2 \ge \alpha \norm u _V^2 \quad (u \in V, t \in [0,T]).
\end{equation}
If $\omega=0$, we say that  the form $\fra$ is  \emph{coercive}. 

For $t \in [0,T]$, a $V$-bounded and quasi-coercive sesquilinear form $\fra(t,.,.)$ is closed.  
The operator   $\A(t)  \in \L(V,V')$  associated with $\fra(t, .,.)$ is defined by
\begin{equation}\label{op1}
\langle \A(t)  u, v \rangle = \fra(t, u,v) \quad {\rm for} \ u,v \in V. 
\end{equation}
We may also associate with $\fra(t,.,.)$ an operator on $H$ by taking the part $A(t)$ 
of $\A(t)$ on $H$; i.e.,\
\begin{align*}
    D(A(t)) := {}& \{ u\in V : \A(t)  u \in H \}\\
    A(t) u := {}& \A(t) u.
\end{align*}
Note that if $\fra(t,.,.)$ is \emph{symmetric}, i.e., 
\[
    \fra(t, u,v)=\overline{\fra(t, v,u)}
\]
for all $u,v \in V$, then  the operator $A(t)$ is self-adjoint.

For a Hilbert space $E$ we denote by $L^2(0,T,E)$ the $L^2$-space on $(0,T)$ of functions with values in $E$ and by $H^k(0,T,E)$ we denote the usual Sobolev space of order $k$ of functions on $(0,T)$ with values in $E$. For $u \in H^1(0,T;E)$ we denote the first derivative by $\dot u$ and for $u \in H^2(0,T;E)$ the second derivative by $\ddot u$.

We start with the  following differentiation result. 
\begin{lemma}\label{lem:differentiation}
Let 
\[
    \fra \colon [0, T] \times V \times V \to \K
\]
be  a $V$-bounded, quasi-coercive non-autonomous form. Suppose that it is Lipschitz  with Lipschitz constant $\dot M$, that is
$$| \fra(t, \phi, \psi) - \fra(s, \phi, \psi) | \le \dot M | t - s |\norm{\phi}_{V} \norm{\psi}_{V}, \ t, s \in [0,T] \ {\rm and } \  \phi, \psi \in V.$$
Let  $u,v  \in H^1(0,T;V)$.
Then $\fra(.,u,v) \in W^{1,1}(0,T)$ and there exists a non-autonomous form $\dot\fra$ which is $V$-bounded with constant $\dot M$ such that
\[
	\fra(.,u,v)\dot{} = \fra(.,u,\dot v) + \fra(.,\dot u,v)+ \dot\fra(.,u,v)
\]
If additionally $\fra$ is symmetric then
\[
	\fra(.,u,u)\dot{} = 2 \Re\fra(.,u,\dot u) + \dot\fra(.,u,u).
\]
\end{lemma}
Note that for $u,v \in V$ we have $\frac{\d}{\d t} \fra(t,u,v) = \dot \fra(t,u,v)$ for a.e.\ $t \in [0,T]$.

This lemma is a consequence of the  next two results.
 \begin{lemma}\label{lemma:int_by_part0}
	Let $u \in H^1(0,T;V)$ and $v \in H^1(0,T;V')$. 
	Then $\langle v(.), u(.) \rangle \in W^{1,1}(0,T)$ and
	\[
			\langle v(.), u(.) \rangle \dot{} = \langle \dot v(.), u(.) \rangle + \langle v(.), \dot u(.) \rangle.
	\]
\end{lemma}
\begin{proof}
	By Fubini's Theorem we have
	\begin{align*}
		\int_0^t \langle \dot v(s), u(s)\rangle \ \d s
			&= \int_0^t \Big\langle \dot v(s), u(0) + \int_0^s \dot u(r) \ \d r \Big\rangle \ \d s\\
			&= \langle v(t), u(0) \rangle - \langle v(0), u(0) \rangle + \int_0^t \int_0^s \langle \dot v(s), \dot u(r) \rangle \ \d r \ \d s\\
			&= \langle v(t), u(0) \rangle - \langle v(0), u(0) \rangle + \int_0^t \int_r^t \langle \dot v(s), \dot u(r) \rangle \ \d s \ \d r\\
			&= \langle v(t), u(0) \rangle - \langle v(0), u(0) \rangle + 
					\int_0^t \langle v(t), \dot u(r) \rangle - \langle v(r), \dot u(r) \rangle \ \d r\\
			&= \langle v(t), u(t) \rangle - \langle v(0), u(0) \rangle - 
					\int_0^t \langle v(r), \dot u(r) \rangle \ \d r.
	\end{align*}
	Thus
	\[
			\langle v(t), u(t) \rangle = \langle v(0), u(0) \rangle + \int_0^t \langle \dot v(s), u(s)\rangle \ \d s+ \int_0^t \langle v(s), \dot u(s)\rangle \ \d s
	\]
	which proves the claim.
\end{proof}
\begin{proposition}\label{prop:lipschitz_continuous_operators}
  Let  $S\colon [0,T]\to \L(V,V')$ be Lipschitz continuous.
    Then the following  assertions hold.
    \begin{enumerate}[label={\rm \alph*)}]
        \item There exists a bounded, strongly measurable function
            $\dot S \colon [0,T] \to \L(V,V')$ such that
            \[
                \frac \d {\d t} S(t)u = \dot S(t)u \quad (u \in V)
            \]
            for a.e.\ $t\in [0,T]$ and
            \[
                \norm{\dot S(t)}_{\L(V,V')} \le L \quad (t \in [0,T])
            \]
            where $L$ is the Lipschitz constant of $S$.
        \item If $ u\in H^1(0,T;V)$, then
            $Su := S(.)u(.) \in H^1(0,T;V')$ and
            \begin{equation}\label{eq:chain_rule}
                (Su)\dot{} = \dot S(.)u(.)+ S(.) \dot u(.).
            \end{equation}
    \end{enumerate}
\end{proposition}
Proposition \ref{prop:lipschitz_continuous_operators} is proved in \cite{ADLO}. Lemma \ref{lem:differentiation}  follows from \eqref{op1}, Lemma \ref{lemma:int_by_part0}
and Proposition \ref{prop:lipschitz_continuous_operators}. 

We shall need the following representation result  due to Lions. See  {\cite[p.\ 156]{Lio59}, \cite[p.\ 61]{Lio61}} or \cite{ADLO}.
\begin{theorem}[Lions' Representation Theorem]\label{thlions}
Let $\mathcal H$ be a Hilbert space, $\V$ a pre-Hilbert space
    such that $\V \hookrightarrow \mathcal H$.
    Let $E\colon \mathcal H \times \V \to \K$
    be sesquilinear such that
    \begin{enumerate}
        \item[{\rm 1)}] for all $w\in \V$, $E(., w)$  is a continuous linear functional on $\mathcal H$;
        \item[{\rm 2)}] $\abs{E(w,w)} \ge \alpha \norm{w}_{\V}^2$ for all $w \in \V$
    \end{enumerate}
    for some $\alpha > 0$. Let $L \in \V'$.
    Then there exists $u\in \mathcal H$ such that
    \[
        L w = E(u,w)
    \]
    for all $w \in \V$.
\end{theorem}
\section{Maximal Regularity for the Damped Wave Equation in $V'$}

Let $H, V$ be Hilbert spaces such that $V \stackrel d \hookrightarrow H$.
We define the following maximal regularity space
\begin{align*}
	\MR(V,V,V') :={}& L^2(0,T,V) \cap H^1(0,T;V) \cap H^2(0,T;V')\\
	 ={}& H^1(0,T;V) \cap H^2(0,T;V'). 
\end{align*}

Let $\fra \colon [0,T] \times V \times V \to \C$ and $\frb \colon [0,T] \times V \times V \to \C$ be  non-autonomous $V$-bounded and quasi-coercive sesquilinear forms. We denote by $\A(t)$ and $\B(t)$ their associated operators in the sense of \eqref{op1}. 
The following is our  first result.  
\begin{theorem}\label{thm:MR_in_V'_damped}
	For every $u_0 \in V$, $u_1 \in H$ and $f \in L^2(0,T;V')$,
	there exists a unique solution 
	$u \in \MR(V,V,V')$ of the 
	non-autonomous second order Cauchy problem
	\begin{equation}\label{eq:SO_CP_damped}
		\left\{
		\begin{aligned} 
			&\ddot u(t) + \B(t)\dot u(t) + \A(t)u(t) = f(t) \quad t\text{-a.e.}\\
			&u(0)=u_0,\,  \dot u(0) = u_1.
		\end{aligned}
		\right.
	\end{equation}
	Moreover there exists a constant  $C >0$ such that 
	\begin{equation}\label{eq:MR_estimate_damped}
		\norm{u}_{\MR(V,V,V')} \le C \Big[ \norm{u_0}_V + \norm{u_1}_H + \norm{f}_{L^2(0,T;V')} \Big].
	\end{equation}
\end{theorem}
As  mentioned in the introduction, this theorem was first proved by Lions \cite[p.\ 151]{Lio61} under an additional regularity assumption on $t \mapsto \fra(t,u,v)$ and $t \mapsto \frb(t,u,v)$. This regularity assumption was removed  in Dautray-Lions \cite[p.\ 667]{DL88}, but taking $f \in L^2(0, T, H)$ and considering mainly symmetric forms (they allow  some  non-symmetric perturbations).  Their proof is based on a Galerkin method. Another proof of 
Theorem \ref{thm:MR_in_V'_damped} was given recently by Batty, Chill and Srivastava  \cite{BCS} but they consider only the case $u_0 = u_1 = 0$.  
Our proof is based on Theorem \ref{thlions} and is in the spirit of  
 Lions \cite{Lio61}. It is   different from the proofs in \cite{DL88} and  \cite{BCS}.  
 
A  classical result of Lions says that
\begin{equation}\label{eq:embedding_in_continuous_functions}
    \MR(V,V') :=  L^2(0,T,V) \cap H^1(0,T;V') \hookrightarrow C([0,T];H),
\end{equation}
and also that for $u\in \MR(V,V')$ the function $\norm{u(.)}^2_H$ is in $W^{1,1}(0,T)$ with
\begin{equation}\label{eq:derivative_of_norm_in_H}
    (\norm{u}^2_H)\dot{} = 2 \Re \langle\dot u, u \rangle, 
\end{equation}
see \cite[p.\ 106]{Sho97} and \cite[p.570]{DL88}. 
This implies that $\MR(V,V,V') \hookrightarrow C([0,T];V)\cap C^1([0,T];H)$.
 Thus for $u \in \MR(V,V,V')$, both $u(0)$ and $\dot u (0)$ make sense. 

We start with the  following basic  lemma.
\begin{lemma}\label{lem:estimates_damped}
	For $v \in H^1(0,T;V)$ we have
	\begin{align*}
		\Big( \int_0^T \norm{v(t)}_V^2 \ \d t\Big)^{1/2}
			\le T \Big(\int_0^T \norm{\dot v(s)}^2_V \ \d s\Big)^{1/2}
				+ \sqrt T \norm{ v(0)}_V.
	\end{align*}
\end{lemma}
\begin{proof}
	Note that $v(t)= v(0)+\int_0^t \dot v(s) \ \d s$, thus
	\begin{align*}
		\int_0^T \norm{v(t)}_V^2 \ \d t 
			&= \int_0^T \Big( v(0) + \int_0^t \dot v(s) \ \d s \,\Big\vert\, v(t) \Big)_V \ \d t\\
			&= \int_0^T \int_s^T (\dot v(s) \mid v(t))_V \ \d t\ \d s
				+\int_0^T (v(0) \mid v(t))_V \ \d t\\
			&\le \int_0^T \int_0^T \norm{\dot v(s)}_V \norm{v(t)}_V \ \d t\ \d s
				+ \int_0^T \norm{ v(0)}_V\norm{v(t)}_V \ \d t\\
			&\le \int_0^T \norm{v(t)}_V \ \d t \ \Big(\int_0^T \norm{\dot v(s)}_V \ \d s
				+ \norm{ v(0)}_V \Big)\\
			&\le \Big(T \int_0^T \norm{v(t)}_V^2 \ \d t\Big)^{1/2} 
				\Big(  \Big(T \int_0^T \norm{\dot v(s)}^2_V \ \d s\Big)^{1/2}
				+ \norm{ v(0)}_V \Big). \tag*{\qedhere}
	\end{align*}
\end{proof}

\begin{proof}[Proof of Theorem \ref{thm:MR_in_V'_damped}]
It suffices to show that there exists a unique solution in the case where  $T<T_0$ 
and  $T_0>0$ is a constant that depends only on the constants $M, \omega$ and $\alpha$ of \eqref{eq:a_continuous} and \eqref{eq:H-elliptic}.
Indeed we can extend this solution to $[0,T]$ for any fixed $T$ as follows. We write the interval $[0,T]$ as a finite union of sub-intervals $[\tau_i, \tau_{i+1}]$, each has length 
less than $T_0$. On each interval $[\tau_i, \tau_{i+1}]$ we have a unique solution $u^{i}$ with $u^{i}(\tau_i)  \in V$, $\dot u^{i}(\tau_i)  \in H$ 
and $u^{i} \in \MR(V,V,V') \hookrightarrow C^1([\tau_i,\tau_{i+1}];H) \cap C([\tau_i,\tau_{i+1}];V)$. On $[\tau_{i+1}, \tau_{i+2}]$ we solve the equation with 
$u^{i+1}(\tau_{i+1}) = u^{i}(\tau_{i+1})$  and $\dot u^{i+1}(\tau_{i+1}) = \dot u^{i}(\tau_{i+1})$. We define $u$ on $[0,T]$ by  $u = u^{i}$ on $[\tau_i, \tau_{i+1}]$ and check easily that 
$u \in \MR(V,V,V')$ (on $[0,T]$)  is the unique solution to \eqref{eq:SO_CP_damped}. 

We prove existence of a solution in the case where 
\begin{equation}\label{eq:T}
	T< T_0 = \min \left\{\frac{\alpha^2}{M^2}, \frac{\alpha}{\sqrt{2} M }\right\}.
\end{equation}
Note that we may assume throughout this proof that the forms $\fra$ and $\frb$ are both coercive.  Indeed, set 
$v(t) = e^{wt} u(t)$ then we have 
\begin{multline}\label{coercive}
 \ddot v(t) + \B(t)\dot v(t) + \A(t)v(t)\\
	= e^{wt} \left[ \ddot u(t) + (\B(t) + 2w) \dot u(t) + (\A(t) + w \B(t) + w^2)u(t) \right].
\end{multline}
Since $\fra$ and $\frb$ are quasi-coercive, we  may choose $w$ large enough such that $\frb + 2w$ and $\fra + w \frb + w^2$ are coercive. Note also that $v \in \MR(V,V,V')$ {\it if and only if}
$u \in \MR(V,V,V')$. 

 We define the Hilbert space 
$\mathcal H := H^1(0,T;V)$ endowed with its usual  norm 
$\norm u_{\mathcal H} := \norm u_{H^1(0,T;V)}$ 
and the pre-Hilbert space 
$$\mathcal V := \{ v \in H^2(0,T;V): \dot v(T) = 0\}$$
 with norm $\norm{.}_{\mathcal V} := \norm{.}_{\mathcal H}$. 
Further we define the sesquilinear form $E\colon \mathcal H \times \V \to \C$ by
\begin{align*}
	E(u,v) := &-\int_0^T (\dot u \mid  \ddot v  )_H \ \d t
			+ \int_0^T  \frb(t, \dot u,\dot v) \ \d t\\
			&{}+ \int_0^T  \fra(t,u,\dot v) \ \d t + \fra(0, u(0),v(0))
\end{align*}
and for $u_0 \in V$, $u_1 \in H$ and $f \in L^2(0,T;V')$ we define $F\colon \V \to \C$ by
\[
	F(v):= \int_0^T  \langle f, \dot v \rangle \ \d t + \fra(0,u_0,v(0)) + (u_1\mid \dot v(0))_H.
\]
We claim  that
\begin{enumerate}
	\item[1)] $E(.,v) \in \mathcal H'$ and $F \in \V'$;
	\item[2)] $E$ is coercive; i.e., there exists a $C>0$ such that
		\(
			\abs{E(v,v)} \ge C \norm{v}_{\mathcal H}^2
		\)
		for all $v \in \V$.
\end{enumerate}
Suppose for a moment that  1) and 2) are satisfied. Then we can  apply  Lions's representation theorem  (see Theorem \ref{thlions}) and obtain  $u \in \mathcal H$ such that
\begin{equation}\label{eq:riesz_solution_damped}
	E(u,v) = F(v) \quad  \forall\ v \in \V.
\end{equation}
 We show  that $u$ is a solution of \eqref{eq:SO_CP_damped}. 

Let  $\psi(t) \in \D(0,T)$ and $w \in V$  and choose  $v(t) := \int_0^t \psi(s) \, \d s\, w$.   It follows from  \eqref{eq:riesz_solution_damped} that 
\begin{align*}
	&- \int_0^T \langle \dot u(t), w \rangle \dot \psi(t) dt
		= \int_0^T \langle f(t) - \B(t) \dot u- \A(t) u(t), w \rangle  \psi(t) \ \d t.
\end{align*}
This means that $\dot u \in H^1(0,T;V')$,  hence $u \in \MR(V,V,V')$ and
\begin{equation}\label{eq:CP_satisfied_damped}
	\ddot u(t) + \B(t) \dot u(t) + \A(t)u(t) = f(t) \quad t\text{-a.e.}
\end{equation}
in $V'$.
For  general $v \in \V$,  we use again  \eqref{eq:riesz_solution_damped} and integration by parts  to  obtain 
\begin{align*}
	&(\dot u(0) \mid  \dot v(0))_H + \int_0^T  \langle \ddot u, \dot v \rangle \ \d t
			+ \int_0^T  \frb(t,\dot u,\dot v) \ \d t
			+ \int_0^T  \fra(t,u,\dot v) \ \d t \\
			&\qquad+ \fra(0, u(0),v(0))
	= \int_0^T \langle f, \dot v \rangle \ \d t + \fra(0,u_0,v(0)) + (u_1 \mid  \dot v(0))_H. 
\end{align*}
This equality  together with \eqref{eq:CP_satisfied_damped} imply  that
\begin{align*}
	(\dot u(0) \mid  \dot v(0))_H + \fra(0, u(0),v(0)) = \fra(0,u_0,v(0)) + (u_1\mid  \dot v(0))_H.
\end{align*}
Since $v \in \V$ is  arbitrary we obtain that $u(0) =u_0$ and $\dot u(0) = u_1$. Therefore,  $u$ is a solution of \eqref{eq:SO_CP_damped} on $[0,T]$ 
for $T \le T_0$ and $T_0$ is such that the above properties 1) and 2) are satisfied. 

Now we return to  1) and 2). Property 1) is obvious. We show the coercivity property 2). 
 Let $v \in \V$. 
The equality  $\frac \d{\d t} \norm{\dot v(t) }_H^2 
		=   2 \Re (\ddot v(t) \mid \dot v(t))_H$ implies
\[\int_0^T  \Re ( \ddot v \mid \dot v )_H\ \d t =  -\frac 1 2 \norm{\dot v(0)}^2_H.
\]
It follows that 
 \begin{align*}
	\abs{E(v,v)} &\ge \Re E(v,v)\\
	&= \frac 1 2 \norm{\dot v(0)}_H^2 + \int_0^T \Re\frb(t,\dot v,\dot v) \ \d t\\
		 &\quad + \int_0^T \Re \fra(t,v,\dot v) \ \d t + \Re\fra(0, v(0),v(0)). 
		\end{align*}
We use coercivity of $\frb, \fra$ and $V$-boundedness of $\fra$ to obtain 
 \begin{align*}
 	\abs{E(v,v)}   &\ge \frac 1 2 \norm{\dot v(0)}_H^2 
	+ \alpha \int_0^T  \norm{\dot v}_V^2 \ \d t\\
	&\quad- M \int_0^T \norm{v}_V \norm{\dot v}_V\ \d t + \alpha \norm{v(0)}_V^2.
	\end{align*}
Therefore, by Young's inequality, we have  
\begin{align*}
 	\abs{E(v,v)} 	&\ge \frac 1 2 \norm{\dot v(0)}_H^2 
	+ \frac \alpha 2 \int_0^T  \norm{\dot v}_V^2 \ \d t
	- \frac {M^2}{2 \alpha} \int_0^T  \norm{v}_V^2\ \d t + \alpha \norm{v(0)}_V^2.
	\end{align*}
Next we apply Lemma~\ref{lem:estimates_damped} to obtain 
	$$\abs{E(v,v)}  \ge \left(\frac \alpha 2 - \frac{M^2 T^2}{\alpha}\right) \int_0^T \norm{\dot v}_V^2 \ \d t + \left(\alpha - \frac{M^2 T}{\alpha}\right) \norm{v(0)}_V^2.$$
Now we use   \eqref{eq:T} and  the fact  that by  Lemma~\ref{lem:estimates_damped}, 
$\int_0^T \norm{v}_V^2 \, \d t$ is dominated (up to a constant) by $ \int_0^T \norm{\dot v}_V^2 \, \d t +  \norm{v(0)}_V^2$.   We obtain  2). 

Next we prove  uniqueness.  Suppose that $u$ and $v$ are two solutions of \eqref{eq:SO_CP_damped} which are in $\MR(V,V,V')$. Set $w = u -v$. Clearly 
$w \in \MR(V,V,V')$ and satisfies (in $V'$) 
$$ \ddot w(t) + \B(t)\dot w(t) + \A(t)w(t) = 0, \quad w(0) = 0, \, \dot w (0) = 0.$$
We show that $w=0$.
For fixed $t \in (0,T]$ we have
$$\int_0^t \Re \langle  \ddot w, \dot w \rangle \ \d s + \int_0^t \Re \frb(s, \dot w, \dot w) \ \d s +  \int_0^t \Re \fra(s,  w, \dot w) \ \d s = 0.$$
Using \eqref{eq:derivative_of_norm_in_H} we have 
$$\int_0^t \Re \langle  \ddot w, \dot w \rangle \ \d s = \frac{1}{2} \int_0^t \big(\norm{\dot w}_H^2\big)\dot{} \ \d s = \frac{1}{2} \norm{\dot w(t)}_H^2 - \frac{1}{2} \norm{\dot w(0)}_H^2 
= \frac{1}{2} \norm{\dot w(t)}_H^2,$$
and hence
 \begin{align*}
 0 & = \frac{1}{2} \norm{\dot w(t)}_H^2 + \int_0^t \Re \frb(s, \dot w, \dot w) \ \d s +  \int_0^t \Re \fra(s,  w, \dot w) \ \d s\\
 &\ge  \frac{1}{2} \norm{\dot w(t)}_H^2 + \alpha \int_0^t \norm{ \dot w}_V^2 \ \d s  - M \int_0^t \norm{w}_V \norm{ \dot w}_V \ \d s.
 \end{align*}
 Here we used coercivity of $\frb$ and $V$-boundedness of $\fra$. Therefore, by  Lemma~\ref{lem:estimates_damped}, we have 
 \begin{align*}
 0 &\ge  \frac{1}{2} \norm{\dot w(t)}_H^2 + \alpha \int_0^t \norm{ \dot w}_V^2 \ \d s  - M \Big(\int_0^t \norm{ w}_V^2 \ \d s\Big)^{1/2}\Big(\int_0^t \norm{ \dot w}_V^2 \ \d s\Big)^{1/2}\\
 & \ge \frac{1}{2} \norm{\dot w(t)}_H^2 + (\alpha - M T) \int_0^t \norm{ \dot w}_V^2 \ \d s.
 \end{align*}
By \eqref{eq:T} we obtain that $w = 0$.  
This shows uniqueness.

Finally, in order to prove the apriori estimate \eqref{eq:MR_estimate_damped}, we consider the operator
  $$ S \colon V \times H \times L^2(0,T, V') \mapsto \MR(V,V,V'), \quad (u_0, u_1, f) \mapsto u.$$
  This is a linear operator which is  well defined thanks to the uniqueness of the solution $u$ of  \eqref{eq:SO_CP_damped}. It is easy to see that 
  $S$ is a closed operator. Therefore it is continuous by the closed graph theorem. This gives \eqref{eq:MR_estimate_damped}. 
 \end{proof}
 The previous proof does not give any information on  the constant $C$ in \eqref{eq:MR_estimate_damped}.  For small 
 time $T$ one can  prove that $C$ depends only on the constants of the forms. This observation will be needed  in  our application to a quasi-linear problem. 
 \begin{proposition}\label{pro-est}
 If $T > 0$ is small enough, then the constant $C$ in \eqref{eq:MR_estimate_damped} depends only on the constants $w$, $\alpha$, $M$ and $T$.
 \end{proposition}
 \begin{proof} Let  $u \in \MR(V,V,V')$ be the   solution  of \eqref{eq:SO_CP_damped}. 
For fixed $t \in (0,T]$ we have
\[
\int_0^t \Re \langle f, \dot u \rangle \ \d s=\int_0^t \Re \langle \ddot u, \dot u \rangle \ \d s + \int_0^t \Re \frb(s, \dot u, \dot u) \ \d s +  \int_0^t \Re \fra(s,  u, \dot u) \ \d s.
\]
Since by \eqref{eq:derivative_of_norm_in_H}
\[
\int_0^t \Re \langle  \ddot u, \dot u \rangle \ \d s = \frac{1}{2} \int_0^t \big(\norm{\dot u}_H^2\big)\dot{} \ \d s = \frac{1}{2} \norm{\dot u(t)}_H^2 - \frac{1}{2} \norm{\dot u(0)}_H^2,
\]
it follows by  Young's inequality that
 \begin{align*}
  &\frac 1 \alpha \int_0^t \norm f_{V'}^2 \ \d s + \frac \alpha 4 \int_0^t \norm{\dot u}_V^2 \ \d s 
	\ge\int_0^t \norm f_{V'} \norm{\dot u}_V \ \d s \ge \int_0^t \Re \langle f, \dot u \rangle \ \d s \\
	&\quad=\frac{1}{2} \norm{\dot u(t)}_H^2  - \frac{1}{2} \norm{\dot u(0)}_H^2
		+ \int_0^t \Re \frb(s, \dot u, \dot u) \ \d s +  \int_0^t \Re \fra(s, u, \dot u) \ \d s\\
 	&\quad\ge  - \frac{1}{2} \norm{\dot u(0)}_H^2
		+ \alpha \int_0^t \norm{ \dot u}_V^2 \ \d s  - M \int_0^t \norm{u}_V \norm{ \dot u}_V \ \d s\\
	&\quad\ge  - \frac{1}{2} \norm{\dot u(0)}_H^2
		+ \frac {3\alpha} 4 \int_0^t \norm{ \dot u}_V^2 \ \d s  - \frac{M^2}{ \alpha} \int_0^t \norm{u}_V^2 \ \d s.
 \end{align*}
 Here we used coercivity of $\frb$ and $V$-boundedness of $\fra$. Therefore, by Lemma \ref{lem:estimates_damped}, we have 
 \begin{equation}\label{eq:MR_d_est}
\begin{split}
\frac 1 \alpha \int_0^t \norm f_{V'}^2 \ \d s + \frac{1}{2} \norm{\dot u(0)}_H^2 
&\ge \frac \alpha 2 \int_0^t \norm{ \dot u}_V^2 \ \d s
	- \frac{M^2}{ \alpha} \int_0^t \norm{ u}_V^2 \ \d s\\
 & \ge  \left(\frac\alpha 2 - \frac{ t^2(2M^2+\alpha)} \alpha \right) \int_0^t \norm{ \dot u}_V^2\ \d s\\
	&\quad -t\left(\frac{2M^2 + \alpha}{\alpha}\right) \norm{u(0)}_V^2 + \frac 1 2\int_0^t \norm{u}_V^2 \ \d s.
 \end{split}
\end{equation}
where we choose $t$ such that $\frac\alpha 2 > \frac{ t^2(2M^2+\alpha)} \alpha$.
Finally, since
\[
	\ddot u(s) = f(s) - \A \dot u(s) - \B u(s) \quad s\text{-a.e.}
\]
we obtain that
\[
	\norm{\ddot u(s)}^2_{V'} \le 3\norm{f(s)}_{V'}^2 + 3M \norm{\dot u(s)}_V^2 + 3M \norm {u(s)}_V^2 \quad s\text{-a.e.}
\]
This together with \eqref{eq:MR_d_est}  ends the proof of the proposition when $T$ is such that \[\frac\alpha 2 > \frac{ T^2(2M^2+\alpha)} \alpha.\tag*{}\]
 
 \end{proof}

\section{Maximal Regularity for the Damped Wave Equation in $H$}
Let $V, H$ be separable Hilbert spaces such that $V \underset d
\hookrightarrow H$ and let 
\[
    \fra, \frb\colon [0,T] \times V \times V \to \K
\]
be closed non-autonomous sesquilinear forms on which we impose the following conditions. 
Each can  be written as the sum of two non-autonomous forms
\[
    \fra(t,u,v)= \fra_1(t,u,v)+ \fra_2(t,u,v),  \quad   \frb(t,u,v)= \frb_1(t,u,v)+ \frb_2(t,u,v)  \quad u,v \in V
\]
where
\[
    \fra_1, \frb_1 \colon [0,T] \times V \times V \to \K
\]
satisfy the following assumptions 
\begin{enumerate}[label={\alph{*})}]
    \item $\abs{\fra_1(t,u,v)} \le M  \norm u_V \norm v_V$ for all $u,v \in V$, $t \in [0,T]$;
    \item $ \fra_1(t,u,u) \ge \alpha \norm u _V^2$ for all $u \in V$, $t \in [0,T]$ with $\alpha > 0$;
    \item $\fra_1(t,u,v) = \overline{\fra_1(t,v,u)}$  for all $u,v \in V$, $t \in [0,T]$;
    \item$\fra_1$ is piecewise Lipschitz-continuous; i.e., there exist $0= \tau_0 < \tau_1 < \dots < \tau_n =T$ such that 
\[
	\abs{\fra_1(t,u,v) - \fra_1(s,u,v)} \le \dot M \abs{t-s} \norm u_V \norm v_V
    	\]
    	for all $u,v \in V,$ $s,t \in [\tau_{i-1},\tau_i]$, $i \in \{1,\dots,n\}$,
\end{enumerate}
and similarly for $\frb_1$. Of course we may  choose  the same constants $M, \dot M$ and $\alpha$ for both forms $\fra_1$ and $\frb_1$. 
We may also  choose that same sub-intervals $0= \tau_0 < \tau_1 < \dots < \tau_n =T$ for both forms. 

The non-autonomous forms
\[
    \fra_2, \frb_2 \colon [0,T] \times V \times V\to \K
\]
are measurable and satisfy 
\begin{enumerate}[label={(\alph*)}]
    \item[e)] $\abs{\fra_2(t,u,v)} \le M \norm u_V \norm v_H$ for all $u, v \in V$, $t \in [0,T]$, 
\end{enumerate}
and similarly for $\frb_2$.

Note that by Lemma \ref{lem:differentiation}, if $\frc$ is a Lipschitz form on $[0,T]$, we may  define its  derivative $\dot \frc(t,.,.)$  and we have 
\begin{equation}\label{deriv}
\abs{\dot \frc(t,u,v)} \le \dot M  \norm u_V \norm v_V, \ u, v \in V
\end{equation}
for some constant $\dot M$.  We shall use this estimate for $\frc = \fra_1$ and for  $\frc = \frb_1$ on sub-intervals of $[0,T]$ where these forms are supposed to be Lipschitz. 

Let us denote by  $\A(t)$ and $\B(t)$ the operators  given by $\langle \A(t) u, v \rangle = \fra(t,u,v)$ and 
$\langle \B(t) u, v \rangle = \frb(t,u,v)$ for all $u, v \in V$.  

As in the previous section we consider  the damped wave equation.  Here we study the  maximal regularity property in  $H$ rather than  in $V'$. 
We  introduce the maximal regularity space
$$\MR(V,V,H) :=  H^1(0,T;V) \cap H^2(0,T;H).$$
We have  
\begin{theorem}\label{thm:MR_in_H}
	Let $\fra = \fra_1+\fra_2$ and $\frb = \frb_1+\frb_2$ be  non-autonomous $V$-bounded and quasi-coercive 
	forms satisfying the above properties $a)-e)$. 
	Then for every $u_0, u_1 \in V$ and $f \in L^2(0,T;H)$,
	there exists a unique solution 
	$u \in \MR(V,V,H)$ of the 
	non-autonomous second order Cauchy problem
	\begin{equation}\label{eq:SO_CP_H}
		\left\{
		\begin{aligned}
			&\ddot u(t) +\B(t)\dot u(t) + \A(t)u(t) = f(t) \quad t\text{-a.e.}\\
			&u(0)=u_0, \dot u(0) = u_1
		\end{aligned}
		\right.
	\end{equation}
	Moreover $\dot u(t) \in V$ for all $t \in [0,T]$. 
\end{theorem}
For a  related result see Lions \cite[p.\ 155]{Lio61}. However the result proved there is restricted to $u_1 = 0$ and assumes $f, f' \in L^2(0,T, H)$. Our proof resembles  that of 
Theorem \ref{thm:MR_in_V'_damped}  and uses similar ideas as in Lions \cite{Lio61}.

We use the following lemma for the proof of Theorem~\ref{thm:MR_in_H}.
\begin{lemma}\label{lem:estimates_H} Suppose that the forms  $\fra_1$ and $\frb_1$   are  Lipschitz continuous on $[0,T]$. 
	Let $v \in H^2(0,T;V)$ and $ \epsilon >0$. Then
	\begin{align*}
		\text{\rm{(i)} } &\int_0^T e^{-\lambda t} \Re \frb_1(t,\dot v,\ddot v) \ \d t 
			= \tfrac \lambda 2 \int_0^T e^{- \lambda t} \frb_1(t,\dot v, \dot v) \ \d t
			- \tfrac 1 2 \int_0^T e^{- \lambda t} \dot\frb_1(t, \dot v, \dot v) \ \d t\\ 
		&\hspace{2cm} +\tfrac 1 2 e^{- \lambda T} \frb_1(T,\dot v(T),\dot v(T))-\tfrac 1 2 \frb_1(0,\dot v(0),\dot v(0)).\\
		\text{\rm{(ii)} } &\int_0^T e^{-\lambda t} \Re \fra_1(t,v,\ddot v) \ \d t =
			\tfrac \lambda 2 e^{-\lambda T} \fra_1(T,v(T),v(T)) -\tfrac \lambda 2 \fra_1(0,v(0),v(0))\\ 
			&\hspace{2cm}+ e^{-\lambda T} \Re\fra_1(T,v(T),\dot v(T)) - \Re\fra_1(0,v(0),\dot v(0))\\
			&\hspace{2cm}+ \tfrac {\lambda^2} 2 \int_0^T e^{-\lambda t} \fra_1(t, v, v) \ \d t
			- \tfrac \lambda 2 \int_0^T e^{-\lambda t}  \dot\fra_1(t, v, v) \ \d t\\
			&\hspace{2cm}-  \int_0^T e^{-\lambda t} \Re\dot\fra_1(t, v, \dot v) \ \d t
			- \int_0^T e^{-\lambda t} \fra_1(t, \dot v,\dot v) \ \d t.\\
		\text{\rm{(iii)} } &\int_0^T e^{-\lambda t} \Re \frb_1(t,\dot v,\ddot v) \ \d t + \int_0^T e^{-\lambda t} \Re \fra_1(t,v,\ddot v) \ \d t\\
			&\hspace{1.6cm}\ge \tfrac 1 2 (\alpha \lambda - 2\dot M - 2M) \int_0^T e^{-\lambda t} \norm{\dot v}_V^2 \ \d t\\
			 &\hspace{2cm}+(\tfrac \lambda 2 (\alpha \lambda - \dot M) - \tfrac{\dot M^2}{2}) \int_0^T e^{-\lambda t} \norm{v}_V^2 \ \d t\\
			&\hspace{2cm}+\tfrac 1 2 e^{-\lambda T} \left[ (\alpha-\epsilon) \norm{\dot v(T)}^2_V + (\lambda \alpha- \tfrac{\dot M^2}{\epsilon}) \norm{v(T)}^2_V \right]\\
			&\hspace{2cm}-\tfrac 1 2 \frb_1(0,\dot v(0),\dot v(0)) -\tfrac \lambda 2 \fra_1(0,v(0),v(0)) - \Re \fra_1(0,v(0),\dot v(0)). 
	\end{align*}
\end{lemma}
\begin{proof}
	The proof of (i) and (ii) is based on Lemma~\ref{lem:differentiation} and the product rule.
	
	Part (i) is a direct consequence of the formulae
	\begin{align*}
		\big(e^{-\lambda t} \frb_1(t,\dot v,\dot v)\big)\dot{}
		&= - \lambda e^{-\lambda t} \frb_1(t,\dot v,\dot v) 
			+ e^{-\lambda t} \dot\frb_1(t,\dot v,\dot v)
			+ 2 e^{-\lambda t} \Re \frb_1(t,\dot v,\ddot v).
	\end{align*}
	
	For (ii) we first calculate the following derivatives
	\begin{align*}
		\Re \big(e^{-\lambda t} \fra_1(t, v, v)\big)\dot{}
		&= - \lambda e^{-\lambda t} \fra_1(t, v, v) + e^{-\lambda t} \Re \dot\fra_1(t, v, v)\\
			&\quad+ 2 e^{-\lambda t} \Re \fra_1(t, v,\dot v)\\
		\Re \big(e^{-\lambda t} \fra_1(t, v, \dot v)\big)\dot{}
		&= - \lambda e^{-\lambda t} \Re\fra_1(t, v, \dot v) 
			+ e^{-\lambda t} \Re\dot\fra_1(t, v, \dot v)\\
			&\quad+ e^{-\lambda t} \fra_1(t, \dot v,\dot v)
			+ e^{-\lambda t} \Re\fra_1(t,  v,\ddot v) 
	\end{align*}
	then we multiply the first equation by $\frac \lambda 2$ and add the second equation.
	Now (ii) follows by integration over $t$ from $0$ to $T$.

	For (iii) we add (i) and (ii) and use coercivity of $\fra_1, \frb_1$ and $V$-boundedness of $\fra_1, \dot \fra_1, \frb_1, \dot \frb_1$.
	Thus
	\begin{align*}
		\int_0^T e^{-\lambda t}& \Re \frb_1(t,\dot v,\ddot v) \ \d t + \int_0^T e^{-\lambda t} \Re \fra_1(t,v,\ddot v) \ \d t\\
			&\ge \tfrac 1 2 (\alpha \lambda - \dot M - 2M) \int_0^T e^{-\lambda t} \norm{\dot v}_V^2 \ \d t\\
			&\quad +\tfrac \lambda 2 (\alpha \lambda - \dot M) \int_0^T e^{-\lambda t} \norm{v}_V^2 \ \d t
			 -  \dot M \int_0^T e^{-\lambda t} \norm{v}_V \norm{\dot v}_V \ \d t\\
			&\quad+\tfrac 1 2 e^{-\lambda T} \left[ \alpha \norm{\dot v(T)}^2_V + \lambda \alpha \norm{v(T)}^2_V - 2 M \norm{v(T)}_V \norm{\dot v(T)}_V \right]\\
			&\quad -\tfrac 1 2 \frb_1(0,\dot v(0),\dot v(0)) -\tfrac \lambda 2 \fra_1(0,v(0),v(0)) - \Re \fra_1(0,v(0),\dot v(0)).
			\end{align*}
We apply Young's inequality and see that the last term is bounded from below by 
			\begin{align*}
			& \tfrac 1 2 (\alpha \lambda - 2\dot M - 2M) \int_0^T e^{-\lambda t} \norm{\dot v}_V^2 \ \d t
			 +( \tfrac \lambda 2 (\alpha \lambda - \dot M)  - \tfrac{\dot M^2}{2}) \int_0^T e^{-\lambda t} \norm{v}_V^2 \ \d t\\
			&\quad+\tfrac 1 2 e^{-\lambda T} \left[ (\alpha-\epsilon) \norm{\dot v(T)}^2_V + 
			(\lambda \alpha- \tfrac{\dot M^2}{\epsilon}) \norm{v(T)}^2_V \right]\\
			&\quad -\tfrac 1 2 \frb_1(0,\dot v(0),\dot v(0)) -\tfrac \lambda 2 \fra_1(0,v(0),v(0)) - \Re \fra_1(0,v(0),\dot v(0))
	\end{align*}
	for $\epsilon >0$.
\end{proof}

\begin{proof}[Proof of Theorem \ref{thm:MR_in_H}]
Uniqueness follows from Theorem~\ref{thm:MR_in_V'_damped} and we only need to prove existence of a solution. As in the proof of 
Theorem~\ref{thm:MR_in_V'_damped} we may assume that the forms $\fra$ and $\frb$ are both coercive (see \eqref{coercive}). \\

\noindent {\it 1- Lipschitz-continuous forms}. Suppose first that the forms $\fra_1$ and $\frb_1$ are Lipschitz-continuous on $[0,T]$. 

 We define the Hilbert space 
\[
	\mathcal H := \{ u \in H^2(0,T;H) \cap H^1(0,T;V) : u(0), \dot u(0), \dot u(T) \in V \}
\]
with norm $\norm u_{\mathcal H}$ given by 
\[
	\norm u_{\mathcal H}^2 := \norm {\ddot u}^2_{L^2(0,T;H)} 
		+ \norm{u}^2_{H^1(0,T;V)} + \norm{u(0)}^2_V + \norm{\dot u(0)}_V^2 + \norm{\dot u(T)}_V^2
\]
and the pre-Hilbert space $\mathcal V := H^2(0,T;V)$ with norm $\norm{.}_{\mathcal V} := \norm{.}_{\mathcal H}$. 
Next we define the sesquilinear form $E\colon \mathcal H \times \V \to \C$ by
\begin{align*}
	E(u,v) := &\int_0^T e^{-\lambda t} (\ddot u \mid \ddot v)_H \ \d t
			+ \int_0^T e^{-\lambda t} \frb(t,\dot u,\ddot v) \ \d t
			+ \int_0^T e^{-\lambda t} \fra(t,u,\ddot v) \ \d t\\
		&{} + \eta (\dot u(0) \mid \dot v(0))_V + \eta (u(0) \mid v(0))_V, 
\end{align*}
where  $\lambda $ and $\eta$  are positive parameters. Later on, we will  choose them to be large enough.  
For $u_0, u_1 \in V$ and $f \in L^2(0,T;H)$,  we define $F\colon \V \to \C$ by
\[
	F(v):= \int_0^T e^{-\lambda t} ( f \mid \ddot v )_H \ \d t 
			+\eta (u_1 \mid \dot v(0))_V + \eta (u_0 \mid v(0))_V
\]
We proceed as in the proof of Theorem~\ref{thm:MR_in_V'_damped}. Suppose for a moment that  
\begin{enumerate}
	\item[1)] $E(.,v) \in \mathcal H'$ and $F \in \V'$;
	\item[2)] $E$ is coercive; i.e., there exists a $C>0$ such that
		\(
			\abs{E(v,v)} \ge C \norm{v}_{\mathcal H}^2
		\)
		for all $v \in \V$.
\end{enumerate}
Then by Lions's representation theorem there exists  $u \in \mathcal H$ such that
\begin{equation}\label{eq:riesz_solution_H}
	E(u,v) = F(v)
\end{equation}
for all $v \in \V$.  For arbitrary $w \in V$ and $\psi \in \mathcal D (0,T)$ we take  $v(t) = \int_0^t \int_0^s \psi(r) \,\d r\, \d s\, w$. It follows from 
\eqref{eq:riesz_solution_H}  that 
\[
	\ddot u(t) +\B(t)\dot u(t) + \A(t)u(t) = f(t)
\]
in $L^2(0,T;V')$.
This identity applied to \eqref{eq:riesz_solution_H} implies that
\begin{align*}
	\eta (\dot u(0) \mid \dot v(0))_V + \eta (u(0) \mid v(0))_V = \eta (u_1 \mid \dot v(0))_V + \eta (u_0 \mid v(0))_V
\end{align*}
for all $v \in \V$. Hence $u(0) =u_0$ and $\dot u(0) = u_1$. This means that $ u \in \MR(V,V,H)$ is a solution of 
\eqref{eq:SO_CP_H}. 

It remain to prove properties 1) and 2). Again, 1) is obvious and we focus on 2).  Let $v \in \V$.
For $\epsilon \in (0,\alpha)$ set
\begin{align*}
 R &:= 	\eta\norm{\dot v(0)}_V^2 + \eta \norm{ v(0)}_V^2
		+  \tfrac 1 2 e^{-\lambda T} \left[ (\alpha-\epsilon) \norm{\dot v(T)}^2_V + (\lambda \alpha- \tfrac{\dot M^2}{\epsilon}) \norm{v(T)}^2_V \right]\\
&\quad  -\tfrac 1 2 \frb_1(0,\dot v(0),\dot v(0))
		 -\tfrac \lambda 2 \fra_1(0,v(0),v(0))  - \Re \fra_1(0,v(0),\dot v(0)).
\intertext{By the $V$-boundedness of $\fra_1$ and $\frb_1$ we have}
R&\ge \tfrac 1 2 e^{-\lambda T} \left[ (\alpha-\epsilon) \norm{\dot v(T)}^2_V + (\lambda \alpha- \tfrac{\dot M^2}{\epsilon}) \norm{v(T)}^2_V \right] + (\eta - \frac{M}{2}) \norm{\dot v(0)}_V^2\\
	 &\quad + (\eta-  \frac{\lambda M}{2}) \norm{ v(0)}_V^2 - M  \norm{\dot v(0)}_V
	 \norm{ v(0)}_V.\\
\intertext{Young's inequality yields}
R&\ge C_1 \left[ \norm{\dot v(T)}^2_V + \norm{ v(T)}^2_V 	+ \norm{\dot v(0)}^2_V  + \norm{ v(0)}^2_V\right]
\end{align*}
for some $C_1 >0$ provided $\lambda$ and $\eta$ are sufficiently large.
Now
\begin{align*}
	\Re E(v,v)&
	=\int_0^T e^{-\lambda t} \norm{\ddot v}^2_H \ \d t
			+ \int_0^T e^{-\lambda t} \Re \frb_1(t,\dot v,\ddot v) \ \d t\\
			&\quad + \int_0^T e^{-\lambda t} \Re \frb_2(t,\dot v,\ddot v) \ \d t
			 +  \int_0^T e^{-\lambda t} \Re\fra_1(t, v,\ddot v) \ \d t\\
			& \quad +  \int_0^T e^{-\lambda t} \Re\fra_2(t, v,\ddot v) \ \d t
			 +\eta (\dot v(0) \mid \dot v(0))_V + \eta (v(0) \mid v(0))_V.
	\intertext{We apply assertion  (iii)  of Lemma~\ref{lem:estimates_H}, it follows that} 
	\Re E(v,v) 
	&\ge \int_0^T e^{-\lambda t} \norm{\ddot v}^2_H \ \d t + \int_0^T e^{-\lambda t} \Re \frb_2(t,\dot v,\ddot v) \ \d t\\
	& \quad + \int_0^T e^{-\lambda t} \Re\fra_2(t, v,\ddot v) \ \d t\\
	& \quad+\frac{1}{2}(\alpha \lambda - 2 \dot M - 2 M) \int_0^T e^{-\lambda t} \norm{\dot v}_V^2\ \d t\\
	& \quad  +  \frac 1 2 (\lambda(\alpha \lambda - \dot M) - {\dot M}) \int_0^T e^{-\lambda t} \norm{v}_V^2\ \d t+ R.
	\end{align*}
	Thus $V$-boundedness of $\fra_2$ and $\frb_2$ and Young's inequality yield
	$$\Re E(v,v) 
	 \ge C \norm{v}_{\mathcal H}^2,$$
for some $C >0$ provided that $\lambda$ and $\eta$ are sufficiently large. This proves 2).

Finally, we have seen that the unique solution $u$ satisfies $\dot u (T) \in V$ but we may replace in the previous arguments $[0,T]$ by 
$[0,t]$ for any fixed $t \in (0,T)$ and obtain $\dot u(t) \in V$. \\

\noindent{\it 2- Piecewise Lipschitz-continuous forms}. Suppose now that the forms $\fra_1$ and $\frb_1$ satisfy assumption d). We may replace in the first step the interval $[0,T]$ by $[\tau_{i-1}, \tau_i]$. There exists  a solution $u^i \in H^1(\tau_{i-1}, \tau_i;V)\cap H^2(\tau_{i-1}, \tau_i;H)$ 
of the equation 
$$\ddot v (t) + \B(t) \dot u(t)  + \A(t)u(t) = f(t) \ \text{a.e.\ } t \in [\tau_{i-1}, \tau_i],$$
with prescribed $u^i (\tau_{i-1}), \dot u^{i}(\tau_{i-1})$ in $V$. We also know from the previous step that 
$u^i (\tau_{i}), \dot u^{i}(\tau_{i}) \in V$. Now we can solve the previous  equation on $[\tau_{i}, \tau_{i+1}]$ and obtain a solution $u^{i+1}$ such that 
 $u^{i+1}(\tau_i) = u^i (\tau_{i})$ and $\dot u^{i+1}(\tau_i) = \dot u^i (\tau_{i})$. We define $u$ on $[0,T]$ by  $u = u^{i}$ on $[\tau_{i-1}, \tau_i]$. It is easy to check that 
 $u \in  \MR(V,V,H)$ and $u $ is a solution to \eqref{eq:SO_CP_H}. 
 This finishes the proof of the theorem. 
\end{proof}


\section{The Wave Equation}

Let $H, V$ be Hilbert spaces such that $V \stackrel d \hookrightarrow H$.
Suppose $\fra \colon [0,T] \times V \times V \to \C$ is a Lipschitz-continuous, symmetric, V-bounded and quasi-coercive non-autonomous form.
We denote again by $\A(t)$ the operator associated with $\fra(t)$ on $V'$ and by $A(t)$ the part of $\A(t)$ in $H$.

We introduce the maximal regularity space 
\[
	\MR(V,H,V') := L^2(0,T;V) \cap H^1(0,T;H) \cap H^2(0,T;V')
\]
for the second order Cauchy problem. We have the following result. 
\begin{theorem}\label{thm:MR_in_V'}
	There exists a unique solution 
	$u \in \MR(V,H,V')$ of the 
	non-autonomous second order Cauchy problem
	\begin{equation}\label{eq:SO_CP}
		\left\{
		\begin{aligned}
			&\ddot u(t)  + \A(t)u(t) = f(t) \quad t\text{-a.e.}\\
				&u(0)=u_0, \dot u(0) = u_1
		\end{aligned}
		\right.
	\end{equation}
	for every $u_0 \in V$, $u_1 \in H$ and $f \in L^2(0,T;H)$.
	Moreover $u(t) \in V$ for all $t \in [0,T]$.  
\end{theorem}
Note that by \cite[p.\ 579]{DL88}, for every $u \in \MR(V,H,V')$,  $\dot u $ can be viewed as a continuous function
from $[0,T]$ into the interpolation space $(H,V')_{\frac{1}{2}}$. In particular, $\dot u(0)$ is well defined and $\dot u (0) \in V'$. 
 
We start with  the following lemma.  Here $ \dot\fra(t,.,.)$ denotes  the derivative of $t \mapsto \fra(t,.,.)$.  
\begin{lemma}\label{lem:estimates}
	Let $v \in H^2(0,T;V)$ with $\dot v(T)=0$. Then
	\begin{align*}
		&\text{\rm{(i)} } \int_0^T e^{- \lambda t} \Re ( \ddot v \mid \dot v )_H\ \d t
			= \frac \lambda 2 \int_0^T e^{- \lambda t} \norm{\dot v }_H^2\ \d t 
				-\frac 1 2 \norm{\dot v(0)}^2_H\\
		&\text{\rm{(ii)} } \int_0^T e^{-\lambda t} \Re \fra(t,v,\dot v) \ \d t = \frac \lambda 2 \int_0^T e^{- \lambda t} \fra(t,v,v) \ \d t
	- \frac 1 2 \int_0^T e^{- \lambda t} \dot\fra(t,v,v) \ \d t\\ 
		&\hspace{4.5cm}+\frac 1 2 \fra(T,v(T),v(T))-\frac 1 2 \fra(0,v(0),v(0))
	\end{align*}
\end{lemma}
\begin{proof}
	For the first part we calculate the formula
	\begin{align*}
		\big(e^{-\lambda t} \norm{\dot v}_H^2\big)\dot{}
		&= -\lambda e^{-\lambda t} \norm{\dot v}_H^2 
			+ 2 e^{-\lambda t} \Re (\ddot v, \dot v)_H.\\
	\intertext{For (ii) we use Lemma~\ref{lem:differentiation} and the product rule to obtain}
		\big(e^{-\lambda t} \fra(t,v,v)\big)\dot{}
		&= - \lambda e^{-\lambda t} \fra(t,v,v) 
			+ 2 e^{-\lambda t} \Re \fra(t,v,\dot v) + e^{-\lambda t} \dot\fra(t,v,v).
	\end{align*}
	Now the Lemma follows by integrating over $t$.
\end{proof}

\begin{proof}[Proof of Theorem \ref{thm:MR_in_V'}]
First we prove existence of a solution. We define the Hilbert space 
$\mathcal H := \{ u \in L^2(0,T;V) \cap H^1(0,T;H) : u(0), u(T) \in V \}$ with norm $\norm u_{\mathcal H}$ such that 
 $\norm u_{\mathcal H}^2 := \norm u^2_{L^2(0,T;V)} + \norm{\dot u}^2_{L^2(0,T;H)} + \norm{u(0)}^2_V+\norm{u(T)}^2_V$ and the pre-Hilbert space
  $\mathcal V := \{ v \in H^2(0,T;V): \dot v(T) = 0\}$ with norm $\norm{.}_{\mathcal V} := \norm{.}_{\mathcal H}$. 
Further we define $E\colon \mathcal H \times \V \to \C$ by
\begin{align*}
	E(u,v) := &-\int_0^T \big(\dot u \,\big\vert\, (e^{-\lambda t} \dot v)\dot{}\, \big)_H \ \d t\\
			&{}+ \int_0^T e^{-\lambda t} \fra(t,u,\dot v) \ \d t + \fra(0, u(0),v(0))
\end{align*}
and for $u_0 \in V$, $u_1 \in H$ and $f \in L^2(0,T;V')$ we define $F\colon \V \to \C$ by
\[
	F(v):= \int_0^T e^{-\lambda t} \langle f, \dot v \rangle \ \d t + \fra(0,u_0,v(0)) + (u_1\mid \dot v(0))_H. 
\]
As in the previous sections, we use Lions's representation Theorem.  
Suppose that the assumptions of Lions's  Theorem are satisfied. Then there exists a $u \in \mathcal H$ such that
\begin{equation}\label{eq:riesz_solution}
	E(u,v) = F(v)
\end{equation}
for all $v \in \V$.  For the particular choice of $v(t) := \psi(t) w$ where $\psi \in \D(0,T)$ and $w \in V$ we obtain from  \eqref{eq:riesz_solution} that 
\begin{align*}
	&\int_0^T \langle \dot u, w \rangle (e^{-\lambda t} \dot \psi(t))\dot{} \ \d t
		= \int_0^T \langle f  - \A u, w \rangle e^{-\lambda t} \dot \psi(t) \ \d t.
\end{align*}
This implies that $\dot u \in H^1(0,T;V')$, hence $u \in \MR(V,H,V')$ and that
\begin{equation}\label{eq:CP_satisfied}
	\ddot u(t)  + \A(t)u(t) = f(t) \quad t\text{-a.e.}
\end{equation}
Following the proof of Lemma 1 and Theorem 2  in \cite[p.\  571 and  575]{DL88} we can integrate by 
   parts in the  first term of $E(u,v)$  to obtain 
\begin{align*}
	E(u,v) &= \langle \dot u(0),   \dot v(0) \rangle+ \int_0^T e^{-\lambda t} \langle \ddot u, \dot v \rangle \ \d t\\
			&\quad+ \int_0^T e^{-\lambda t} \fra(t,u,\dot v) \ \d t  + \fra(0, u(0),v(0))\\
	&{}= \int_0^T e^{-\lambda t} \langle f, \dot v \rangle \ \d t + \fra(0,u_0,v(0)) + (u_1\mid  \dot v(0))_H
\end{align*}
where we used the identity \eqref{eq:riesz_solution}.  This  together with \eqref{eq:CP_satisfied} implies that
\begin{align*}
	\langle \dot u(0),  \dot v(0) \rangle + \fra(0, u(0),v(0)) = \fra(0,u_0,v(0)) + (u_1\mid \dot v(0))_H.
\end{align*}
Since $v \in \V$ was arbitrary this shows that $u(0) =u_0$ and $\dot u(0) = u_1$.

Next we  check the assumptions of Theorem \ref{thlions}. Assumption  1) is again easy to verify.   Let $v \in \V$, then integration by parts yields to
\begin{align*}
	\abs{E(v,v)} \ge \Re E(v,v)
	&= \norm{\dot v(0)}_H^2 + \int_0^T e^{- \lambda t} \Re ( \ddot v \mid \dot v )_H\ \d t\\
	&\quad+ \int_0^T e^{-\lambda t} \Re \fra(t,v,\dot v) \ \d t + \fra(0, v(0),v(0)).
\end{align*}
Thus Lemma~\ref{lem:estimates} applied to the first and second integral and Young's inequality shows that
\begin{align*}
	\Re{E(v,v)}  &\ge \frac 1 2 \norm{\dot v(0)}_H^2 
	+ \frac \lambda 2 \int_0^T e^{- \lambda t} \norm{\dot v }_H^2\ \d t
	+ \frac \lambda 2 \int_0^T e^{- \lambda t} \fra(t,v,v) \ \d t\\
	&\quad - \frac 1 2 \int_0^T e^{- \lambda t} \dot\fra(t,v,v) \ \d t 
	+ \frac 1 2 \fra(T, v(T),v(T))+ \frac 1 2 \fra(0, v(0),v(0))\\
	&\ge C \norm{v}_{\mathcal H}^2
\end{align*}
for some $C>0$ if $\lambda$ is large enough.
Note that we can choose $C$ depending only on the coercivity, $V$-boundedness, Lipschitz constant of the form and on $T$.

Uniqueness:  
Let $u \in \MR(V,H,V')$ be a solution of \eqref{eq:SO_CP} where $f=0$ and $u_0=u_1=0$.
We have to show that $u=0$.
Fix $r \in [0,T]$ and define $v_r(t) := \int_t^T \1_{[0,r]} u(s) \ \d s$. Then $v_r \in H^1(0,T;V)$ with $v_r(r)=0$ and $\dot v_r = -\1_{[0,r]} u$.
We obtain
\begin{align*}
0 
&= 2 \int_0^T \Re \langle \ddot u, v_r \rangle \ \d t + 2 \int_0^T \Re \fra(t,u,v_r) \ \d t\\
&= 2 \int_0^r \int_t^r \Re\langle \ddot u(t), u(s) \rangle \ \d s \ \d t - 2 \int_0^r \Re\fra(t,\dot v_r, v_r) \ \d t\\
&= 2 \int_0^r \int_0^s \Re\langle \ddot u(t), u(s) \rangle \ \d t \ \d s - 2 \int_0^r \Re\fra(t,\dot v_r, v_r) \ \d t\\
&= 2 \int_0^r \Re\langle \int_0^s \ddot u(t)  \ \d t , u(s) \rangle \ \d s - \int_0^r (\fra(t, v_r, v_r))\dot{} - \dot \fra(t,v_r,v_r) \ \d t\\
&= 2 \int_0^r \Re\langle \dot u , u \rangle \ \d s + \fra(0, v_r(0), v_r(0)) - \int_0^r \dot \fra(t,v_r,v_r) \ \d t\\
&\ge \norm{u(r)}_H^2  + \alpha \norm{v_r(0)}_V^2 - \dot M \int_0^r \norm{v_r}_V^2 \ \d t.
\end{align*}
We set $w(r):= v_r(0) = \int_{0}^r u(s) \ \d s \in L^2(0,T;V)$. Then $w(r)-w(t) = v_r(t)$ and
$$ \alpha \norm{w(r)}_V^2 \le \dot M \int_0^r \norm{w(r)-w(t)}_V^2 \ \d t 
\le 2r\dot M \norm{w(r)}_V^2 + 2\dot M \int_0^r \norm{w(t)}_V^2 \ \d t.$$
Let $0<r_0 < \tfrac{ \alpha}{2 \dot M}$ and set $C_{r_0} := \alpha-2r_0\dot M>0$, then for every $r \in [0,r_0]$ we have
$$  \norm{w(r)}_V^2 \le 2 \dot M C_{r_0}^{-1} \int_0^r \norm{w(t)}_V^2 \ \d t.$$
We  conclude by Gronwall's lemma  that $w(r) = 0$ for all $r \in [0,r_0]$, hence 
$u=0$ on $[0,r_0]$. Now we may proceed inductively to obtain $u = 0$ on $[0,T]$.
\end{proof}

\begin{remark}
If we add a $V\times H$-bounded perturbation to $\fra$ as in Section 4, we can still prove existence in Theorem~\ref{thm:MR_in_V'}.
But for the uniqueness we have to assume additionally that this perturbation is also $H\times V$-bounded.
\end{remark}
\begin{remark} 
Let $B(t)$ be bounded operators on $H$ with $\| B(t) \|_{{\mathcal L}(H)} \le M_B $ for a.e.\ $t \in [0,T]$. We consider the wave equation 
\begin{equation}\label{eq:SO_CP00}
		\left\{
		\begin{aligned}
			&\ddot u(t) + B(t) \dot u(t) + \A(t)u(t) = f(t) \quad t\text{-a.e.}\\
				&u(0)=u_0, \dot u(0) = u_1&
		\end{aligned}
		\right.
	\end{equation}
Then for $u_0 \in V$, $u_1 \in H$ and $f \in L^2(0,T, V')$ there exists  a solution $u \in \MR(V,H,V')$ to \eqref{eq:SO_CP00}. The proof is the same as above, one has only to change $E(u,v)$ into
\begin{align*}
	E(u,v) := &-\int_0^T \big(\dot u \,\big\vert\, (e^{-\lambda t} \dot v)\dot{}\, \big)_H \ \d t\\
			&{}+ \int_0^T e^{-\lambda t} (B(t) \dot u \mid \dot v)_H \ \d t\\
			&{}+ \int_0^T e^{-\lambda t} \fra(t,u,\dot v) \ \d t + \fra(0, u(0),v(0)).
\end{align*}
The uniqueness of $u$ is  however not clear except if  the map $t \mapsto  B(t)$ is Lipschitz. If this later condition is satisfied one can use similar ideas as in \cite[p.\ 686]{DL88} to  prove uniqueness. The proof for uniqueness in Theorem \ref{thm:MR_in_V'} is similar to that of \cite[p.\ 673]{DL88}. 
\end{remark}

\section{Applications}\label{section:applications}
In this section we give applications of  our results. We consider two problems, one is  linear and the second one is  
 quasi-linear.\\

{\it I) Laplacian with time dependent Robin boundary conditions.}\\
Let $\Omega$ be a  bounded domain of $\R^d$ with Lipschitz boundary $\Gamma$. 
Denote by $\sigma$ be the $(d-1)$-dimensional Hausdorff measure on $\Gamma$.
Let  
\[
	\beta_1, \beta_2\colon [0,T]  \times \Gamma \to \R
\]
be bounded measurable functions which are Lipschitz continuous w.r.t.\ the first variable, i.e.,
\begin{equation}\label{lipbeta}
 	\lvert \beta_i(t,x) - \beta_i(s, x) \rvert \le M \lvert t-s\rvert \quad (i=1,2)
\end{equation}
for some constant $M$ and all $t, s \in [0,T], \ x \in \Gamma$. We consider the symmetric forms
\[
	\fra,\frb\colon [0,T] \times H^1(\Omega) \times H^1(\Omega) \to \R
\]
defined by 
\begin{equation}\label{formbeta}
	\fra(t, u, v) = \int_\Omega \nabla u \nabla v\ \d x + \int_\Gamma \beta_1(t, .) u v\ \d\sigma.
\end{equation}
and
\begin{equation}\label{formbeta2}
	\frb(t, u, v) = \int_\Omega \nabla u \nabla v\ \d x + \int_\Gamma \beta_2(t, .) u v\ \d\sigma.
\end{equation} respectively.
The forms $\fra, \frb$ are $H^1(\Omega)$-bounded and quasi-coercive. 
The first statement follows readily from the continuity 
of the trace operator and the boundedness of $\beta$. 
The second one is a consequence of the inequality 
\begin{equation}\label{trace-comp}
\int_\Gamma \lvert u \rvert^2 \ \d\sigma \le \epsilon \norm u_{H^1(\Omega)}^2 + c_\epsilon \norm u_{L^2(\Omega)}^2,
\end{equation}
which is valid for all $\epsilon > 0$ ($c_\epsilon$ is a constant depending on $\epsilon$). 
Note that $\eqref{trace-comp}$ is a consequence of compactness of the trace as an operator from $H^1(\Omega)$ into $L^2(\Gamma, \d \sigma)$, see \cite[Chap.\ 2 § 6, Theorem 6.2]{Nec67}.

Let $\A(t)$ be the operator associated with $\fra(t,.,.)$ and $\B(t)$ the operator associated with $\frb(t,.,.)$.
Note that the part $A(t)$ in $H:= L^2(\Omega)$ of $\A(t)$ is interpreted as 
 (minus) the Laplacian  with  time dependent Robin boundary conditions
\[
	\partial_\nu v + \beta_1(t,.) v = 0 \text{ on } \Gamma.
\]
Here we use the following weak definition of the normal derivative.
Let $v \in H^1(\Omega)$ such that $\Delta v \in L^2(\Omega)$.
Let $h \in L^2(\Gamma, \d \sigma)$.  Then $\partial_\nu v = h$
by definition if
$\int_\Omega \nabla v \nabla w + \int_\Omega \Delta v w = \int_\Gamma h w \, \d \sigma$ for all $w \in H^1(\Omega)$.
Based on this definition, the domain of $A(t)$ is the set
\[
	D(A(t)) = \{ v \in H^1(\Omega) : \Delta v \in L^2(\Omega),  \partial_\nu v + \beta_1(t) v\vert_\Gamma = 0 \},
\]
and for $v\in D(A(t))$ the operator is given by $A(t)v = - \Delta v$.

Maximal regularity on $H$  for the first order Cauchy problem associated with $A(t)$ was proved in \cite{ADLO}. Here we study the second order problem.
 By Theorem \ref{thm:MR_in_H}, the damped wave equation 
\begin{equation*}
\left\{  \begin{aligned}
         & \ddot u(t)  - \Delta \dot u(t) -  \Delta u(t)  = f(t)\\
                       & u(0)    =u_0,  \quad \dot u (0)  = u_1  \in H^1(\Omega)\\
                     &   \partial_\nu (\dot u (t) + u(t))  + \beta_2(t,.) \dot u(t) + \beta_1(t,.) u(t) = 0   \text{ on } \Gamma
                         \end{aligned} \right.
\end{equation*}
    has a unique solution  $u \in \MR(V,V,H) =  H^2(0,T;L^2(\Omega))\cap H^1(0,T;H^1(\Omega))$ whenever  $f \in L^2(0,T, L^2(\Omega))$. 
  
Indeed,  Theorem \ref{thm:MR_in_H} implies that there exists $u \in \MR(V,V,V')$ with $u(0)=u_0$, $\dot u (0)  = u_1$ and
\begin{align}\label{eq:Robin_Laplace}
	( \ddot u, v )_H + \frb(t,\dot u, v)  + \fra(t, u, v) = ( f, v )_H
\end{align}
for all $v \in V$ and all $t\in[0,T] \setminus N$, where $N$ is a Lebesgue null set.
Let $t\in[0,T] \setminus N$, then for the special choice $v \in \D(\Omega)$ we obtain that \eqref{eq:Robin_Laplace} implies 
$\ddot u(t)  - \Delta \dot u(t) -  \Delta u(t)  = f(t)$.
This together with \eqref{eq:Robin_Laplace} and the above definition of the normal derivative shows
\[
	\partial_\nu (\dot u (t) + u(t))  + \beta_2(t,.) \dot u(t) + \beta_1(t,.) u(t) = 0   \text{ on } \Gamma.
\]
    
{\it II)  A quasi-linear problem.}\\
Let $\Omega$ be a bounded open set of $\R^d$ and let $H$ be the real-valued Hilbert space $L^2(\Omega, \d x)$ and $V$ be a closed subspace of $H^1(\Omega)$ which contains $H_0^1(\Omega)$. If $V \not= H^1_0(\Omega)$ we assume that $\Omega$ is a Lipschitz domain to ensure that the embedding 
of  $V$ in $H$ is compact.  This latter property is always true for  $V=H^1_0(\Omega)$ for any bounded domain $\Omega$.

For $g,h \in L^2(0,T;H)$ we define the forms $\fra_{g,h}, \frb_{g,h}\colon [0,T]\times V\times V \to \R$ by
\[
	\fra_{g,h}(t,u,v) = \sum_{k,j=1}^d \int_\Omega a_{jk}(t,x,g,h) \partial_k u \partial_j v \ \d x
\]
and
\[
	\frb_{g,h}(t,u,v) = \sum_{k,j=1}^d \int_\Omega b_{jk}(t,x,g,h) \partial_k u \partial_j v \ \d x.
\]
We assume that the coefficients $a_{jk}, b_{jk} \colon [0,T] \times \Omega \times \R \times \R \to \R$ are uniformly bounded on $[0,T] \times \Omega \times\R\times\R$ by a constant $M>0$
and satisfy the usual ellipticity condition
\[
	\sum_{k,j=1}^d a_{jk}(t,x,y,z) \xi_k \xi_j \ge \eta |\xi |^2,\quad \sum_{k,j=1}^d b_{jk}(t,x,y,z) \xi_k \xi_j \ge \eta |\xi |^2
\]
for a.e.\ $(t,x) \in [0, T] \times \Omega$ and all $y,z \in \R$, $\xi \in \R^d$.
Here $\eta > 0$ is a constant.
Moreover we assume that $a_{jk}(t,x,.,.), b_{jk}(t,x,.,.)$ are continuous for a.e.\ $(t,x)$. 
We denote by $\A_{g,h}(t)$ and $\B_{g,h}(t)$ the associated operators. 

Given $u_0 \in V$, $u_1 \in H$ and $f \in L^2(0,T; V')$ the second order Cauchy problem
 \begin{equation}\label{gh}
	\left\{
	\begin{aligned}
		&\ddot u(t) +\B_{g, h}(t)\dot u(t) + \A_{g,h}(t)u(t) = f(t) \quad t\text{-a.e.} \\
			& u(0)=u_0, \dot u(0) = u_1
	\end{aligned}
	\right.
\end{equation} 
has a unique solution $u_{g,h} \in \MR(V,V,V')$
by Theorem~\ref{thm:MR_in_V'_damped}.
Moreover, by Proposition \ref{pro-est}  there exists $C >0$ and $0 < T_0 \le T$ depending only on $M$ and $\eta$ such that  the solution 
of \eqref{gh} on $[0,T_0]$ satisfies the estimate 
\begin{equation}\label{eq:MR_estimate_quasi_linear}
		\norm{u_{g,h}}_{\MR_{T_0}(V,V,V')} \le C \Big[ \norm{u_0}_V + \norm{u_1}_H + \norm{f}_{L^2(0,T_0;V')} \Big].
	\end{equation}
Note that $C$ and $T_0$ are independent of $g$ and $h$.
We want  to show that the quasi-linear problem
\begin{equation}\label{eq:quasi_linear_damped}
	\left\{
	\begin{aligned}
		& \ddot u(t) +\B_{u, \dot u}(t)\dot u(t) + \A_{u,\dot u}(t)u(t) = f(t) \quad t\text{-a.e.}\\
			& u(0)=u_0, \dot u(0) = u_1
	\end{aligned}
	\right.
\end{equation}
has a solution $u$ in $\MR(V,V,V')$.
We define the mapping $S\colon H^1(0,T;H) \to H^1(0,T;H)$ by $Sg:=u_{g, \dot g}$.
Note that by \eqref{eq:MR_estimate_quasi_linear} and the fact that $C$ is independent of $g$ and $h$, 
 $\Im(S)$ is a bounded subset of $\MR(V,V,V')$.
Moreover  by Aubin-Lions lemma, $\MR(V,V,V')$ is compactly embedded into $H^1(0,T;H)$.  
Therefore,  if $S$ is continuous then we can apply  Schauder's fixed point theorem to obtain  
 $u \in H^1(0,T;H)$ such that $Su=u$. Thus $u$ is also in $\MR(V,V,V')$
and $u \in \Im(S)$. Hence $u$ is a solution of  \eqref{eq:quasi_linear_damped}.

It remains to  prove that $S$ is continuous.
Let $g_n \to g$ in $H^1(0,T;H)$ and set $u_n := Sg_n$. 
Since a sequence converges to a fixed element $u$ if and only if each subsequence has a subsequence converging to $u$ we may deliberately take subsequences.
Since $L^2(0,T;H)$ is isomorphic to $L^2((0,T)\times \Omega)$ we may assume (after taking a sub-sequence) that $g_n \to g$ and $\dot g_n \to \dot g$ for a.e.\ $(t,x)$. Furthermore since the sequence $u_n$ is bounded in $\MR(V,V,V')$ we may assume (after taking a sub-sequence) that 
$u_n \to u$ in $H^1(0,T;H)$ and $u_n \rightharpoonup u$ in $\MR(V,V,V')$.
Hence $a_{jk}(t,x,g_n,\dot g_n) \to a_{jk}(t,x,g,\dot g)$ and 
$b_{jk}(t,x,g_n,\dot g_n) \to b_{jk}(t,x,g,\dot g)$ for a.e.\ $(t,x)$. 
Now the equality $u_n = Sg_n$ means that 
\begin{align*}
	\langle \ddot u_n,& v \rangle_{L^2(0,T;V'),L^2(0,T;V)} 
	+ \sum_{j,k=1}^d ( \partial_j \dot u_n \mid b_{jk}(t,x,g_n,\dot g_n) \partial_k v )_{L^2(0,T;H)} \\
	&+ \sum_{j,k=1}^d ( \partial_j u_n \mid a_{jk}(t,x,g_n, \dot g_n) \partial_k v )_{L^2(0,T;H)} = \langle f, v \rangle_{L^2(0,T;V'),L^2(0,T;V)} 
\end{align*}
for all $v \in L^2(0,T;V)$ and $u_n(0) = u_0$, $\dot u_n(0) = u_1$.
By the dominated convergence  theorem 
$a_{jk}(t,x,g_n, \dot g_n) \partial_k v \to a_{jk}(t,x,g, \dot g) \partial_k v$ in $L^2(0,T;H)$.
Moreover $u_n \rightharpoonup u$ in $\MR(V,V,V')$ implies that
$\partial_j u_n \rightharpoonup \partial_j u$ and 
$\partial_j \dot u_n \rightharpoonup \partial_j \dot u$ in $L^2(0,T;H)$.
Thus taking the limit for $n \to \infty$ yields
\begin{align*}
	\langle \ddot u,& v \rangle_{L^2(0,T;V'),L^2(0,T;V)} 
	+ \sum_{j,k=1}^d ( \partial_j \dot u \mid b_{jk}(t,x,g,\dot g) \partial_k v )_{L^2(0,T;H)} \\
	&+ \sum_{j,k=1}^d ( \partial_j u \mid a_{jk}(t,x,g, \dot g) \partial_k v )_{L^2(0,T;H)} = \langle f, v \rangle_{L^2(0,T;V'),L^2(0,T;V)} 
\end{align*}
for all $v \in L^2(0,T;V)$ and $u(0) = u_0$, $\dot u(0) = u_1$.
Note that for the initial condition we have used that $\MR(V,V,V') \hookrightarrow C^1([0,T];H) \cap C([0,T];V)$, see \eqref{eq:embedding_in_continuous_functions}. 
This is equivalent to $Sg=u$. Hence $S$ is continuous.


\noindent
 \emph{Dominik Dier}, Institute of Applied Analysis, 
University of Ulm, 89069 Ulm, Germany,\\
 \texttt{dominik.dier@uni-ulm.de}

\quad\\
\noindent
\emph{El Maati Ouhabaz,} Institut de Math\'ematiques (IMB), Univ.\  Bordeaux, 351, cours de la Libération, 33405 Talence cedex, France,\\ 
\texttt{Elmaati.Ouhabaz@math.u-bordeaux1.fr}

\end{document}